\documentclass[12pt]{iopart}
\usepackage{iopams}
\usepackage{amssymb}
\usepackage{enumerate}
\usepackage{bbm}
\usepackage{mathrsfs}

\expandafter\let\csname equation*\endcsname\relax
\expandafter\let\csname endequation*\endcsname\relax
\usepackage{amsmath}
\usepackage{colortbl}
\usepackage{amsthm}

\newtheorem{lem}{Lemma}[section]

\newtheorem{pro}[lem]{Proposition}
\newtheorem{thm}[lem]{Theorem}
\newtheorem{rem}{Remark}

\newcommand\hdim{\dim_{\mathcal H}}
\newcommand\pdim{\dim_{\mathcal P}}
\def\L{\mathscr{ L}}
\def\F{\mathcal F}
\def\C{\mathcal C}
\def\G{\mathcal G}
\def\H{\mathcal H}
\def\A{\mathcal A}
\def\K{\mathscr{K} }
\def\R{\mathbb R}
\def\Z{\mathbb Z}
\def\N{\mathbb N}
\def\Q{\mathbb Q}

\def\x{\textbf x}
\def\y{\textbf y}
\def\z{\textbf z}

\begin{document}

\title[Metric results of  the intersection of sets  in Diophantine approximation]{Metric results of  the intersection of sets  in Diophantine approximation }

\author{Chen Tian$^1$ and Liuqing Peng$^2$ }

\vspace{16pt plus3pt minus3pt}

\address{$^1$School of Statistics and Mathematics, Hubei University of Economics, Wu Han, Hubei 430205, People's Republic of China\\$^2$School of Mathematics and Statistics, Hubei University of Education, Wu Han, Hubei 430205, People's Republic of China}

\vspace{16pt plus3pt minus3pt}

\ead{tchen@hbue.edu.cn and liuqingpeng@hue.edu.cn}
\let\thefootnote\relax\footnotetext{This work was supported by NSFC 12171172.}

\begin{abstract}
Let $\psi : \R_{>0}\rightarrow \R_{>0}$ be a non-increasing function. Denote by $W(\psi)$  the set of $\psi$-well-approximable points and by $E(\psi)$ the set of points $x\in[0,1]$ such that for any $0 < \epsilon < 1$  there exist infinitely many $(p,q)\in\Z\times\N $  with
$$\left(1-\epsilon\right)\psi(q)< \left| x-\frac{p}{q}\right|< \psi(q) .$$ In this paper, we investigate  the metric properties of the set $E(\psi).$ Specifically, we compute the $s$-dimensional Hausdorff measure $\H^s(E(\psi))$  of $E(\psi)$ for a large class of  $s \in (0,1].$  Additionally, we establish that
$$\hdim E(\psi_1) \times \cdots \times E(\psi_n) =\min \{ \hdim E(\psi_i)+n-1: 1\le i \le n  \},$$ where   $\psi_i:\R_{> 0}\rightarrow \R_{> 0} $ is a non-increasing function satisfying $\psi_i(x)=o(x^{-2}) $ for $1\le i \le n.$
\end{abstract}

\begin{indented}
\item[]Keywords: Approximation by
rational numbers, Hausdorff measure, Hausdorff dimension, Fractals
\end{indented}

\begin{indented}
\item[]Mathematics Subject Classification numbers: 11J83,  28A80
\end{indented}

\vspace{16pt plus3pt minus3pt}

%
%
\section{Introduction}

In this paper, we focus on determining the metric properties of  the intersection of  sets in  Diophantine  approximation. Our motivation stems from classical Diophantine approximation.
Let $\psi : \R_{>0}\rightarrow \R_{>0}$ be a non-increasing function such that $x\mapsto x^2\psi(x)$ is non-increasing, and consider the set of $\psi$-well-approximable points
$$W(\psi)=\left\{x\in [0,1]: \left| x-\frac{p}{q}\right|< \psi(q),\ \text{i.m.}\ (p,q)\in\Z\times\N \right\} ,$$ where {\em i.m.~}denotes {\em infinitely many} for brevity.

 Khintchine \cite{K24} proved that $W(\psi)$ is a Lebesgue  null set if and only if
$\sum\limits_{n=1}^{\infty}n\psi(n) < \infty.$ Subsequently, Jarn\'{\i}k \cite{J29} and independently Besicovitch  \cite{Bes34} showed that for $\tau\ge 2$
$$\hdim W(x \mapsto x^{-\tau})=\frac{2}{\tau} ,  $$
here and hereafter $\hdim$ denotes the Hausdorff dimension. Furthermore, Jarn\'{\i}k \cite{J31} established the following $s$-dimensional Hausdorff measure  $\H^s( W(\psi))$ of the set $W(\psi)$.
\begin{thm}[Jarn\'{\i}k \cite{J31}]\label{thmj}Let $\psi:\R_{> 0}\rightarrow \R_{> 0} $ be such that $x\mapsto x^2\psi(x)$ is non-increasing and  $0<s<1.$ Assume that the sum $\sum\limits_{n= 1}^\infty n\psi(n)$ converges and that the function $x\mapsto x^2\psi^s(x)$ is non-increasing. Then,
$$   \H^s(W(\psi))=\left\{\begin{array}{ll}
    0 &\text{if}~\sum\limits_{n=1}^\infty n\psi^s(n) < \infty; \\
    \infty &\text{if}~\sum\limits_{n=1}^\infty  n\psi^s(n) =\infty .
    \end{array}
  \right.$$
\end{thm}

In the same paper, Jarn\'{\i}k \cite{J31} confirmed that if $\psi(x)=o(x^{-2})$ then the exact set
 $$\text{Exact}(\psi):=W(\psi)\setminus \bigcup_{0<c<1}W(c\psi) $$
is non-empty. This result was later improved by Bugeaud \cite{B03}, who showed that provided that  the function $x\mapsto x^2\psi(x)$ is non-increasing and $\sum_{n=1}^\infty n\psi(n)<\infty ,$ then
\begin{align}\label{result1}\hdim \text{Exact}(\psi) = \hdim W(\psi) = \frac{2}{\lambda},\end{align}  where $ \lambda= \liminf \limits_{x\rightarrow \infty} \frac{-\log \psi (x)}{ \log x}  .$  Additionally, Bugeaud and Moreira \cite{BM11} proved that the condition can be relaxed to the requirement that $\psi:\R_{> 0}\rightarrow \R_{> 0} $ is a non-increasing function and satisfies $\psi(x)=o(x^{-2}) .$  In \cite{FW22},   Fraser and Wheeler calculated the Fourier dimension of the set $ \text{Exact}(\psi).$ For more results, we refer the reader to \cite{BGN23,BS11,BDV01,B08,M18,S23}.

In this paper, we are interested in a slightly larger set
\begin{align*}E(\psi):=\bigcap_{0<\epsilon<1} \left\{x\in [0,1]: \left(1-\epsilon\right)\psi(q)< \left| x-\frac{p}{q}\right|< \psi(q),\ \text{i.m.}\ (p,q)\in\Z\times\N  \right  \}  ,\end{align*}   which can also be regarded as a type of ``exact"  set. Recall that the set $\text{Exact}(\psi)$  consists of real numbers $x\in[0,1]$ such that
$$ |x-p/q|\le \psi(q) \quad \quad \text {infinitely often} $$
and
$$ |x-p/q|\ge c\psi(q) \quad \quad \text {for any}\ c<1\ \text{and} \ \text{any} \ q\ge q_0(c,x) .$$ Obviously,
$$\text{Exact}(\psi) \subset E(\psi) \subset W(\psi)  .$$ Thus, it follows from \eqref{result1} that
$$\hdim  E(\psi) = \hdim \text{Exact}(\psi) =\hdim W(\psi) = \frac{2}{\lambda}, \quad  \text{where } \lambda= \liminf \limits_{x\rightarrow \infty} \frac{-\log \psi (x)}{ \log x}.  $$ Since the Hausdorff measure of $ \text{Exact}(\psi)$ remains unknown, we are unable to get the result of the Hausdorff measure of $E(\psi)$ as we do for the results with the Hausdorff dimension.   This raises the natural question: What is the Hausdorff measure of the set $E(\psi)$? Is it the same as the Hausdorff measure of $W(\psi)$? The following theorem  provides answers to these questions.


\begin{thm}\label{thmmeasure}Let $\psi:\R_{>0}\rightarrow \R_{>0}$ and $0<s\le 1.$ Assume that the function $x\mapsto x^2\psi^s(x)$ is non-increasing and tends to $0$ as $x$ goes to infinity. Then
$$   \H^s(E(\psi))=\left\{\begin{array}{ll}
    0 &\text{if}~\sum\limits_{n=1}^\infty n\psi^s(n) < \infty; \\
    \H^s((0,1)) &\text{if}~\sum\limits_{n=1}^\infty  n\psi^s(n) =\infty.
    \end{array}
  \right.$$
\end{thm}
\begin{rem}
The set $\text{Exact}(\psi)$ is  closely related to
the set of $\psi$-badly  approximable points,  defined as
  $$\text{Bad}(\psi)=W(\psi) \setminus \bigcap_{c>0}W(c\psi) .$$
We have the inclusion
   $\text{Exact}(\psi) \subset \text{Bad}(\psi) \subset W(\psi).$ More related results can be found in \cite{Ba23,Ko24,S25}.
\end{rem}


The above theorem implies that the set $E(\psi)$  is of nice structure. Then, it will be a rewarding task to find other fine features enjoyed by $E(\psi)$ beyond its own metric theory. Inspired  by a simple observation of Erd\"{o}s \cite{E62} regarding the dimension of the Cartesian product of the set of Liouville numbers, we discover another nice property of $E(\psi),$ i.e. the Hausdorff dimension of its Cartesian product.

More specifically, let
$$\mathfrak{L} :=\bigcap_{\tau>2}\left\{x\in \R: \left|x-\frac{p}{q}\right|<\frac{1}{q^{\tau}}, \ \text{i.m.} \ (p,q)\in\Z\times\N \right\}. $$ In terms of Hausdorff dimension and gauge function, the set of Liouville numbers is relatively small. For instance, $$\hdim \mathfrak{L}=0 , $$ as discussed in \cite{BDK05,Ol05,OR06}. However,  Erd\"{o}s \cite{E62} discovered that each real number can be expressed as the sum of two Liouville numbers. In other words, one has $\R=f(\mathfrak{L} \times \mathfrak{L})$ where $f: \mathfrak{L} \times \mathfrak{L}\rightarrow \R$ is  defined as $f(x,y)=x+y.$ Since $f$ is Lipschitz,  it follows from a fundamental property of Hausdorff dimension that
$$\hdim(\mathfrak{L} \times \mathfrak{L} )  \ge \hdim f(\mathfrak{L} \times \mathfrak{L}) =\hdim \R=1. $$ Similarly, we investigate the Hausdorff dimension of the Cartesian product of $E(\psi).$

Let us recall the well-known results on  Cartesian product sets provided  by Marstrand \cite{Mar54} and Tricot \cite{Tri82}.

\begin{thm}[Marstrand \cite{Mar54};Tricot \cite{Tri82}]\label{thm2} Let $A,~B$ be two Borel measurable sets in $\R^d$ and $\R^m.$  Denote the packing dimension by $\pdim.$ Then
$$ \hdim A+\hdim B\leq \hdim (A\times B)\leq \hdim A+\pdim B.$$
 If $\hdim B =\pdim B,$ then
\begin{equation}\label{des1}  \hdim (A\times B) =\hdim A+\hdim B.\end{equation}
\end{thm}

The above theorem implies that if one of $A$, $B$ satisfies its Hausdorff dimension equal to its  packing dimension, for example, self-similar sets, then the equality \eqref{des1} holds. But $E(\psi)$ is dense in $[0,1]$, so $\hdim E(\psi)<\pdim E(\psi)=1.$ Thus, Theorem \ref{thm2} cannot give an exact dimension of the product of $E(\psi)$. To handle this, we establish the following theorem.

\begin{thm}\label{mainthm}For each $1\leq i \leq n,$ let $\psi_i:\R_{> 0}\rightarrow \R_{> 0} $ be a non-increasing function satisfying $\psi_i(x)=o(x^{-2}) .$ Then we have
$$\hdim E(\psi_1) \times \cdots \times E(\psi_n)=\min_{1\leq i\leq n} \big\{ n-1+ \hdim E(\psi_i)\big\}=\min_{1\leq i\leq n} \left\{n-1+\frac{2}{\lambda_i}\right\},$$
where $\lambda_i= \liminf \limits_{x\rightarrow \infty} \frac{-\log \psi_i (x)}{ \log x}. $
\end{thm}

\begin{rem}
 The assumptions on $\psi_i$ imply $\lambda_i \ge 2.$ Let us notice that Wang and Wu \cite{WW24} determined the exact Hausdorff dimension of $W(\psi_1) \times \cdots \times W(\psi_n)$  when   $x \mapsto  x \psi_i(x)$ is non-increasing  for  $1\leq i \leq n.$
\end{rem}

 The paper is organized as follows. In Section 2, we present some preliminaries. In Section 3, we establish Theorem \ref{thmmeasure}. The last section is devoted to the proof  of Theorem \ref{mainthm}. We end this section with notations.
\begin{itemize}
\item For any $0<s<1$ and a ball $B=B(x,r),$ write $B^s$ for the ball $B(x,r^s);$
  \item Write $\lfloor A \rfloor$ for the largest integer not larger than $A;$
  \item Denote the radius of  a ball $B$ by $|B|$;
  \item Let $\#  E$  stand for the cardinality of a finite set $E;$
  \item Given $c>0,$ for a rectangle $R,$  $cR$ denotes the rectangle with the same center and side lengths scaled by a factor of $c.$ The same notation is also employed  for a ball $B;$
  \item We refer to  two distinct  balls $B,B'$ in $X$ as $C$-separated provided that $$d(B,B'):=\inf\limits_{\x \in B, \ \y \in B'} \min\limits_{1\leq i \leq n}|x_i-y_i| > C ;$$
  \item  $a \asymp b$ if $A^{-1} <a/b<A$ and $a\ll b$ if $a \leq A \cdot b$ for some unspecified constant $A>1.$
\end{itemize}

\section{Preliminaries}
 In this section, we present some essential materials.
\subsection{Hausdorff measures}
 For completeness, we begin by defining  the Hausdorff measure and the Hausdorff dimension. Let $E$ be a subset of a metric
space $(X,d).$ Given $\delta>0,$ a countable (or finite) collection of sets $\{U_i\}$  is called a $\delta$-cover of set $E$ if $E\subset \bigcup_{i}U_i$ and each $U_i$ satisfies $0<\text{diam}(U_i) \le \delta,$ where $\text{diam}(U_i)$ denotes the diameter of $U_i.$  For each $s>0,$ we define
 $$\H^s_\delta(E)=\inf \left\{ \sum\limits_{i=1}^\infty (\text{diam}(U_i))^s: \{U_i\} \ \text{is a $\delta$-cover of $E$} \right\}, $$ where the infimum is taken over all possible $\delta$-covers of $E.$ The {\em{$s$-dimensional Hausdorff measure $\H^s(E)$ of $E$}} is defined as
 $$\H^s(E)=\lim\limits_{\delta \rightarrow  0} \H^s_{\delta}(E).$$ Further, the {\em  Hausdorff dimension  $\hdim E$ of $E$} is defined by
 $$ \hdim  E=\inf \{s\ge0 : \H^s(E)=0\}=\sup \{s\ge0 : \H^s(E)=\infty\}.$$ For further details see \cite{Fa90}.

The following mass distribution principle is the classical method for estimating the lower bound of the Hausdorff dimension.

\begin{pro}[Mass Distribution Principle, \cite{Fa90}]\label{p1}
Let $\mu$ be a probability measure supported on a measurable set $F$ of $(X,d)$. Suppose there are positive constants $c$ and $r_0$ such that
$$\mu(B(x,r))\le c r^s$$
for any ball $B(x,r)$ with radius $r\le r_0$ and center $x\in F$. Then $\mathcal{H}^s(F)\ge 1/{c}$ and so $\hdim F\ge s$.
\end{pro}

\begin{lem}[The $5r$ covering lemma, \cite{H01}]\label{5rcover}
  Every family $\F$ of balls of uniformly bounded diameter in a metric space $(X,d)$ contains a disjoint subfamily $\G$ such that
  $$ \bigcup_{B\in\F}B\subset \bigcup_{B\in\G}5B.$$
\end{lem}

\subsection{Positive and full measure sets}
In this section we state two measure theoretic results and prove some lemmas for later use.

Let $\mu$ be a finite measure supported on metric space $(X,d).$ We call $\mu$ a {\em doubling measure} if there exists a constant $\lambda>1$ such that for $x\in X$
$$\mu(B(x, 2r)) \le  \lambda \mu(B(x, r)).$$
Let $\L$ denote the Lebesgue measure on $\R,$ which is clearly a doubling measure.

\begin{lem}[\cite{BDV06}]\label{lem3}
 Let $(X,d)$ be a metric space and let $\mu$ be a finite doubling
measure on $X$ such that any open set is $\mu$ measurable. Let $E$ be a Borel subset of $X.$ Assume that there are constants $r_0,~c>0$ such that for any ball $B$ with $r(B)<r_0$ and center in $X$, we have that $\mu(E\cap B)\ge c \mu(B).$ Then, for any
ball $B$
$$\mu(E\cap B)=\mu(B) .$$
\end{lem}

\begin{lem}[\cite{BDV06}]\label{lem4}
Let $(X,d)$ be a metric space and $\mu$ be a finite measure on $X.$
Let $B$ be a ball in $X$ and $E_n$ a sequence of $\mu$-measurable sets. Suppose there exists a constant $c>0$ such that $\limsup\limits_{n\rightarrow \infty}\mu(E_n)\ge c \mu(B).$ Then
$$ \mu \left(  B\cap  \limsup_{n\rightarrow \infty} E_n\right) \ge c^2\mu(B).$$
\end{lem}

The following lemma will be extremely useful in the construction of the Cantor subset. From now on,  let $c_k=1-2^{-k}$ for any $k\ge1.$   For any $ \frac{p}{q}\in\Q$ and $\psi : \R_{>0}\rightarrow \R_{>0},$ define
 $$ \C_k\left(\frac{p}{q}, \psi\right)=\left\{x\in[0,1]: c_k\psi(q)<x-\frac{p}{q}<\psi(q)\right\}.$$

\begin{lem}\label{keylem}
Let $\{\psi_l\}_{1\le l \le n}$
be functions from $\R_{>0}$ to $\R_{>0}$ such that there exists an positive integer $Q_0$ satisfying \begin{equation*} q^2\psi_l(q)<1/1000\end{equation*}
for all $q\ge Q_0 $ and  $1\le l \le n.$ Then for any ball $B\subset [0,1]$ and $Q\in\N$ satisfying
$$Q\ge \max \{10000 |B|^{-1} \log (1/|B|),9Q_0\},$$
 there exists a set  $J_B=\{\frac{p_i}{q_i}\}_{1\leq i \leq \left\lfloor\frac{|B|Q^2}{80}\right\rfloor}\subset B$  satisfying the following properties.
\begin{enumerate}
\item For each $\frac{p_i}{q_i} \in J_{B},$ $$\frac{Q}{9}\leq q_i \leq Q ~\text{and}~ \ \gcd(p_i,q_i)=1 ;$$

\item For distinct $\frac{p_i}{q_i} , \frac{p_j}{q_j} \in J_{B}$ and $1\leq l \leq n,$

  \begin{equation}\label{j3}\C_k\left(\frac{p_i}{q_i}, \psi_l\right) \subset B  \quad \text{and} \quad
d\left(\C_k\left(\frac{p_i}{q_i},\psi_l\right), \C_k\left(\frac{p_j}{q_j},\psi_l\right)\right)>  \frac{1}{2Q^2}  .\end{equation}

\end{enumerate}
\end{lem}


\begin{proof}
  It follows from Dirichlet's theorem (see \cite{La95})  that for any $x\in \frac{1}{2}B,$ there exists  $\frac{p}{q}\in \Q$ satisfying
  \begin{equation}\label{d1_new}
    \left| x -\frac{p}{q} \right|<\frac{1}{qQ} \ \text{and} \ 1\leq q \leq Q.
  \end{equation}
Denote by $\mathcal {Q}$ the set of $\frac{p}{q}$ for which \eqref{d1_new} is satisfied by a point $x$ in $\frac{1}{2}B.$
Let \begin{equation*}\label{setA}A= \bigcup_ {\substack{\frac{p}{q} \in \mathcal {Q} \\ 1 \leq q < \frac{Q}{9}}} B\left(\frac{p}{q}, \frac{1}{qQ}\right) . \end{equation*} Then
\begin{equation}\label{d5_new}\L(A) \leq \sum\limits_{q=1}^{\lfloor Q/9\rfloor+1} \frac{4}{qQ}( q|B|+1) < \frac{|B|}{2} ,\end{equation} since $ Q\ge 10000 |B|^{-1} \log (1/|B|).$

Let $\left\{\frac{p_i}{q_i}\right\}_{1\leq i \leq t}$ be a maximal set of rational numbers in $\mathcal{Q}$ that satisfies $ \frac{Q}{9}\leq q \leq Q$ and \begin{equation}\label{j2}\left|\frac{p_i}{q_i}- \frac{p_l}{q_l}\right| \ge 1 /Q^2\end{equation} for $1\leq i \neq l \leq t.$ Then
\begin{equation}\label{d4_new}\left(\frac{1}{2}B\right)\setminus A \subset \bigcup\limits_{i=1}^{t} B(\frac{p_i}{q_i}, \frac{10}{Q^2}) . \end{equation} Indeed,
for each $x \in \left(\frac{1}{2}B \right)\setminus A,$ there exists  $\frac{p}{q} \in \mathcal{Q}$ such that
\begin{equation}\label{d2_new}
   \left| x -\frac{p}{q} \right|<\frac{1}{qQ} \ \text{and} \ \frac{Q}{9}\leq q \leq Q.
\end{equation} If $\frac{p}{q}$ is not in $\left\{\frac{p_i}{q_i}\right\}_{ 1 \le i \leq t},$ then by maximality, there exists $1\leq i \leq t$ such that
\begin{equation}\label{d3_new}
  \left| \frac{p}{q}-\frac{p_i}{q_i} \right| < \frac{1}{Q^2}.
\end{equation}
Combining \eqref{d2_new} and \eqref{d3_new}, and applying the triangle inequality, we obtain \eqref{d4_new}. Applying \eqref{d5_new} and \eqref{d4_new}, we conclude that
 \begin{align*}t \cdot \frac{20}{Q^2} \ge\L\left(\frac{1}{2}B\setminus A\right) & \ge |B|- \L\left(\frac{1}{2}B\cap A \right)  \\
  & \ge \frac{|B|}{4} .
  \end{align*}Hence,
\begin{equation*}
  t\ge \frac{|B|Q^2}{80}.
\end{equation*}

Notice that $B\left(\frac{p_i}{q_i}, \frac{10}{Q^2} \right) \cap \frac{1}{2}B \neq \emptyset $ for each $1\leq i \leq t.$ Thus, by $Q \ge  10000 |B|^{-1} \log (1/|B|)  \ge \sqrt{40/|B|},$ we have
 \begin{equation}\label{da}  B\left(\frac{p_i}{q_i}, \frac{10}{Q^2}\right) \subset B. \end{equation} Since  $q_i\ge \frac{Q}{9} \ge Q_0,$
 \begin{equation}\label{j1}\psi_l(q_i)\leq \frac{1}{1000\cdot q_i^2} \leq \frac{81}{Q^2} \cdot \frac{1}{1000}< \frac{1}{4Q^2}.  \end{equation} Combining \eqref{da}, \eqref{j1} and \eqref{j2}, we conclude \eqref{j3}, thereby completing the proof.
\end{proof}
The following lemma will be needed in  Section \ref{section1}. Since the length of $\mathcal{C}_{k}\left(\frac{p}{q},\psi\right)$ is $\frac{\psi(q)}{2^{k }},$ it can be expressed as the ball centered at $y$ with radius $\frac{\psi(q)}{2^{k+1 }},$ and thus
$$\C_{k}^s\left(\frac{p}{q},\psi\right)=B\left(y, \frac{\psi^s(q)}{2^{(k+1)s}} \right).$$ Naturally, we have
$$ \left|\C_{k}^s\left(\frac{p}{q},\psi\right)\right|=\left|\C_{k}\left(\frac{p}{q},\psi\right)\right|^s=\frac{\psi^s(q)}{2^{(k+1)s}} .$$

\begin{lem}\label{lem2}
 Given a function
$\psi:\R_{>0}\rightarrow \R_{>0},$  let
$s\in(0 ,1]$ be a real number such that the sum $\sum\limits_{n= 1}^\infty n\psi^s(n)$ diverges and the function $x\mapsto x^2\psi^s(x)$ is non-increasing and tends to $0$ as $x$ goes to infinity. Then for any ball $B\subset [0,1]$ and positive integers $Q',k,$  there exists a set $R_{B,Q',k}=\left\{\frac{v_i}{u_i}\right\}_{1\le i \le t}\subset B$ satisfying the following properties.
  \begin{enumerate}
    \item \label{cite1}For any $1\le i\le t,$ $$u_i\ge Q', \  3\C_{k}^s\left(\frac{v_i}{u_i},\psi\right)\subset B;$$
    \item \label{cite2}For any $1\le  i \ne j \le t,$
        $$ 3\C_{k}^s\left(\frac{v_i}{u_i},\psi\right) \cap 3\C_{k}^s\left(\frac{v_j}{u_j},\psi\right)=\emptyset;$$
    \item  \label{cite3}We have \begin{align*}
      \sum_{i=1}^t\L\left(\C_{k}^s\left(\frac{v_i}{u_i},\psi \right) \right)\ge 10^{-7}\L(B).
    \end{align*}

  \end{enumerate}

\end{lem}

\begin{proof} Since $\psi^s(x)=o(x^{-2})$ where $0<s\le 1,$  there exists a positive integer $Q_0$ such that $ q^2\psi(q)<1/1000$ holds
for all $q\ge Q_0 .$ Let  $Q$ be an integer satisfying
\begin{equation}\label{s0} Q \ge \max  \left\{20000 |B|^{-1} \log (2/|B|),20000 |B|^{-1} ,9Q_0, 9Q'\right\}.\end{equation}

It follows from the non-increasing property of $x \mapsto x^2\psi^s(x)$  that
$$ \sum\limits_{n=1}^\infty n\psi^s(n) =\sum\limits_{k=0}^\infty \sum\limits_{Q^k \leq 9n < Q^{k+1}} n\psi^s(n) \asymp \sum\limits_{k=0}^\infty Q^{2k} \psi^s(Q^k/9). $$ Thus, by $ \sum\limits_{n=1}^\infty n \psi^s(n)=\infty,$
\begin{equation}\label{s1}
  \sum\limits_{k=0}^\infty Q^{2k} \psi^s(Q^k/9)=\infty.
\end{equation} Let $l\ge1$ be an integer such that
\begin{equation}\label{s2}
  x^2\psi^s(x)\le 10^{-6}\cdot 2^{(1+k)s}\  \text{and} \ \psi^s(x) \ge \psi(x) \ \text{for any } \ x\ge Q^l/9.
\end{equation}Combining \eqref{s1} with \eqref{s2}, we can assert that there exists a  positive integer $N$ with
\begin{equation}\label{s3}
  \frac{1}{128} \le 2^{-(1+k)s} \sum_{h=l}^{l+N}Q^{2h} \psi^s(Q^h/9)\le \frac{1}{64}.
\end{equation}

By \eqref{s0}, we can apply Lemma \ref{keylem} to $Q^l$ and $\frac{1}{2}B.$ Then we get a set $\A_0$ of rational numbers satisfying
\begin{itemize}
  \item $ \# \A_0 = \left\lfloor \frac{|B|Q^2}{160} \right\rfloor;$
  \item for any $\frac{u}{v}\in \A_0,$
  $$ \frac{Q^l}{9} \le u \le Q,\quad  \left|\C_{k}\left(\frac{v}{u},\psi\right)\right|=\frac{\psi(u)}{2^{k+1}} \quad \text{and} \quad   \C_{k}\left(\frac{v}{u},\psi\right)\subset \frac{1}{2}B; $$
  \item for distinct $\frac{v}{u},~\frac{v'}{u'} \in \A_0, $
  $$ d\left(\C_{k}\left(\frac{v}{u},\psi\right),\C_{k}\left(\frac{v'}{u'},\psi\right) \right)> \frac{1}{2Q^{2l}}.$$
\end{itemize}Using \eqref{s0} and \eqref{s2}, we get that
$$ \frac{\psi^s(u)}{2^{(k+1)s} }\le \frac{81}{10^6\cdot Q^{2l}} <\min\left\{\frac{1}{24Q^{2l}}, \frac{1}{8}|B| \right\}. $$
Thus $$3\C_{k}^s\left(\frac{v}{u},\psi\right)\subset B \quad \text{and}  \quad    d\left(3\C_{k}^s\left(\frac{v}{u},\psi\right),3\C_{k}^s\left(\frac{v'}{u'},\psi\right) \right)\ge \frac{1}{4Q^{2l}}.$$ Hence, $$
 3\C_{k}^s\left(\frac{v}{u},\psi\right) \cap 3\C_{k}^s\left(\frac{v'}{u'},\psi\right)=\emptyset. $$

Repeat the above process with $Q^{l+j}$ for any $1\le j\le N.$ Keep in mind that we  hope to end up with a collection of far separated balls. Thus, we have to drop some balls in each step. We now proceed by induction. Suppose that for some $0\le j \le N-1,$  we have constructed sets $\A_0, \cdots, \A_j$ of rational numbers such that for any $\frac{a}{b}\neq \frac{c}{d}\in \bigcup\limits_{h=0}^{j}\A_h $
$$3\C_{k}^s\left(\frac{a}{b},\psi \right) \cap  3\C_{k}^s\left(\frac{c}{d},\psi \right) =\emptyset. $$

Applying Lemma \ref{keylem} to $Q^{l+j+1}$ and $\frac{1}{2}B,$ we can get a set $\A_{j+1}'$ of rational numbers satisfying
\begin{itemize}
  \item $\# \A'_{j+1} = \left \lfloor \frac{|B|Q^{2(l+j+1)} }{160} \right\rfloor;$
  \item for any $\frac{v}{u} \in \A'_{j+1},$
  $$ Q^l/9 \le u \le Q,\quad  \left|\C_{k}^s\left(\frac{v}{u},\psi\right)\right|=2^{-(k+1)s}\psi^s(u) \quad \text{and} \quad   \C_{k}^s\left(\frac{v}{u},\psi\right)\subset B; $$
  \item for distinct $\frac{v}{u},~\frac{v'}{u'} \in \A'_{j+1}, $
  $$ d\left(3\C_{k}^s\left(\frac{v}{u},\psi\right),3\C_{k}^s\left(\frac{v'}{u'},\psi\right) \right)\ge \frac{1}{4Q^{2(l+j+1  )}}.$$
\end{itemize}Note that for any $0 \leq h \le j$ we have constructed at most $\left\lfloor\frac{|B|Q^{2(l+h)}}{160}\right\rfloor$ balls of the form $\C_{k}^s\left(\frac{p}{q},\psi\right)$ with  radius  at most equal to $ 2^{-(k+1)s}\psi^s(Q^{l+h}/9).$  Let
$$N_{j+1} =\#\left\{\frac{u}{v}\in \A'_{j+1}\colon 3\C_{k}^s\left(\frac{u}{v} ,\psi\right)\cap 3\C_{k}^s\left(\frac{a}{b} ,\psi\right)\neq \emptyset \ \text{for some} \ \frac{a}{b} \in  \bigcup\limits_{h=0}^{j}\A_h \right\}. $$ Thus by \eqref{s0} and \eqref{s3}
\begin{align*}
  N_{j+1}\le & \sum\limits_{h=0}^j \frac{ |B| Q^{2(l+h)}}{160} \left(  {3\cdot 2^{3-(k+1)s}\psi^s(Q^{l+h}/9)Q^{2(l+j+1)} }+2\right) \\
  \le& \frac{|B|Q^{2(l+j+1)}}{80} \left(3\cdot2^{2-(k+1)s}\sum\limits_{h=0}^N Q^{2(l+h)}\psi^s(Q^{l+h}/9)+ \sum\limits_{h=1}^\infty Q^{-2h}\right)\\
  \le & \frac{|B|Q^{2(l+j+1)}}{320}.
\end{align*}Further, by $\# \A'_{j+1} = \left \lfloor \frac{|B|Q^{2(l+j+1)}}{ 160}\right\rfloor,$ there are  at least $\left\lfloor \frac{|B|Q^{2(l+j+1)}}{400}\right\rfloor+1$  points in $A'_{j+1}$ such that the corresponding balls $3\C_{k}^s\left(\frac{u}{v} ,\psi\right)$ are pairwise disjoint and  do not intersect any of the balls $3\C_{k}^s\left(\frac{a}{b} ,\psi\right)$ constructed previously. We write $\A_{j+1}$ for the set of these points. Finally, the induction step is finished.

Combining these, the total  Lebesgue measure of  all  balls in the form of $\C^s_{k}\left(\frac{v}{u},\psi\right)$ for $\frac{v}{u}\in \bigcup_{j=0}^N \A_j$ is at least
\begin{equation*}
  L= \frac{2|B|}{400}\sum\limits_{h=0}^N 2^{-(k+1)s}Q^{2(l+h)}\psi^s(Q^{l+h}).
\end{equation*}
Since \eqref{s2}, \eqref{s3} and $x\mapsto x^2\psi^s(x)$  is non-increasing,
\begin{align*}
  L\ge&\frac{\L(B)}{32400} \sum\limits_{h=0}^N 2^{-(k+1)s}Q^{2(l+h+1)}\psi^s(Q^{l+h+1}/9)\\
  \ge&\frac{\L(B)}{32400}\left(\frac{1}{128} - 2^{-(k+1)s}Q^{2l}\psi^s(Q^{l}/9) \right)\\
  \ge& \frac{\L(B)}{32400} \left( \frac{1}{128} -\frac{1}{1000}\right) \ge 10^{-7}\L(B),
\end{align*} which completes this proof.
\end{proof}

\begin{lem}\label{lem5}
Let $\psi: \R_{>0}\rightarrow \R_{>0}$ be such that $x\mapsto x^2\psi(x)$ is non-increasing and $\psi(x)=o(x^{-2}).$ Then, we have $\L(E(\psi))=1$ if the sum
$$\sum\limits_{n=1}^\infty n\psi(n)$$
diverges, and $\L(E(\psi))=0$ otherwise.
\end{lem}
\begin{proof}
  The convergence part follows from the Khintchine theorem. Thus we only need to prove the divergence part. For any ball $B$ and integer $Q',$  define
  $$E_{Q'}:=\bigcup_{k\ge Q'}\bigcup_{\frac{v}{u} \in R_{B,Q',k}}\C_k\left( \frac{v}{u},\psi\right) .$$ Combining  Lemma \ref{lem2} \eqref{cite1} with the fact that $\#R_{B,Q',k}<\infty,$ we deduce that
  $$\limsup\limits_{Q'\rightarrow\infty } E_{Q'} \subset  B\cap E(\psi).$$
  It follows from Lemma \ref{lem2} \eqref{cite2} and \eqref{cite3} that $\L(E_{Q'}) \ge 10^{-7} \L(B)$ which together with Lemma \ref{lem4} implies that $ \L(\limsup_{Q'\rightarrow \infty} E_{Q'})\ge 10^{-14} \L(B).$  Hence, $\L(B\cap E(\psi))\ge 10^{-14} \L(B).$
Since the measure $\L$ is a doubling measure, the lemma follows directly from Lemma \ref{lem3}.
\end{proof}

\section{Proof of Theorem \ref{thmmeasure}}\label{section1}
This section aims to prove Theorem \ref{thmmeasure}. By Lemma \ref{lem5}, it suffices to consider the case $0<s<1$ with  $x\mapsto x^2\psi^s(x)$ non-increasing and $\psi^s(x)=o(x^{-2}).$

The convergence part in Theorem \ref{thmmeasure} follows directly from Theorem \ref{thmj} (indeed, this is a natural covering argument); the primary challenge lies in the divergence part. To prove the divergence part,  for any constant $\eta >0,$ we construct a Cantor subset $K_{\eta} \subset E(\psi)$ and a probability measure $\nu$ supported on $K_{\eta}.$ This measure $\nu$ satisfies  the scaling property:
\begin{equation}\label{goal}\nu(B(x,r)) \ll \frac{r^s}{\eta} \end{equation}
for any ball $B(x,r)$ with sufficiently small radius $r.$
The implied constant in the $\ll$ is independent of $B(x,r)$ and $\eta$. Thus, by the Mass Distribution Principle (Proposition \ref{p1}), we get
$$\H^s(E(\psi))\ge \H^s(K_{\eta}) \gg \eta .$$ Since $\eta$ can be chosen arbitrarily large, we conclude that
$$\H^s(E(\psi))=\infty=\H^s((0,1)), $$ which completes the proof of Theorem \ref{thmmeasure}.

 Now, we are in a position  to construct the Cantor set $K_{\eta}$ which supports a measure $\nu$ satisfying the desired property.


\subsection{Cantor subset construction}
In this section, we specify the construction of the Cantor set  $K_{\eta}$. We focus on the construction of the first level of $K_{\eta}$, and the other levels can be done similarly. Let $B_0=(0,1).$ Set  $\K_0=\{B_0\}$  and $K_0=B_0.$

{\em The first level $K_1$}

The first level $\K_1(B_0)$ consists of a family of sub-levels $\{\K_1(B_0,l) \}_{l\ge1}.$ We start by defining $\K_1(B_0,1).$

\textbullet \ Step $1_1.$ Notation $U(1,B_0,1).$

Let $$ U(1,B_0,1)=\{B_0\}=\{(0,1)\} .$$ It is trivial  that
$$\sum_{B\in U(1,B_0,1)} \L(B) \asymp |B_0| .$$ Then we will apply Lemma \ref{lem2} to each ball in $U(1,B_0,1).$

\textbullet \  Step $2_1.$  Use Lemma \ref{lem2}.

Let $G_{1,1}$ be  sufficiently large so that
\begin{equation}\label{size1}
 \psi(q)<\psi^s(q)\quad\text{and}\quad \frac{\L\left(\C_1\left(\frac{p}{q},\psi\right)\right) }{\L\left(\C_1^s\left(\frac{p}{q},\psi\right)\right)}<\frac{1}{24} \quad \text{whenever} \quad q\ge G_{1,1}.
\end{equation}This is possible since $0<s<1$  and $\psi^s(x)=o(x^{-2}).$  Set $Q'= G_{1,1}$ and $k=1.$ For each $B\in U(1,B_0,1),$ applying Lemma \ref{lem2} to $B$ yields a collection of well-separated  balls of the form
$$3\C_1^s\left(\frac{p_i}{q_i},\psi\right)\subset B \quad  \text{with}\quad \frac{p_i}{q_i} \in R_{B,G_{1,1},1},$$ which satisfy
$$3\C_1^s\left(\frac{p_i}{q_i},\psi\right) \cap 3\C_1^s\left(\frac{p_j}{q_j},\psi\right)=\emptyset \quad \text{if} \ i\neq j ,$$ and
$$ \L(B) \ge \sum_{\frac{p_i}{q_i}  \in R_{B,G_{1,1},1} }\L\left(\C_1^s\left(\frac{p_i}{q_i},\psi\right)\right)\ge 10^{-7}\L(B) .$$

\textbullet \ Step $3_1.$ Shrinking. Notation $A(1,B,G_{1,1})$ and $K_1(B_0,1).$

For each $B\in U(1,B_0,1),$  we can assert that $$3\C_1\left(\frac{p_i}{q_i},\psi\right)\subset 3\C_1^s\left(\frac{p_i}{q_i},\psi\right)\subset B \quad \text{with} \quad \frac{p_i}{q_i} \in R_{B,G_{1,1},1}.$$
Denote by
$$A(1,B,G_{1,1})=\left\{\C_1\left(\frac{p}{q},\psi\right): \frac{p}{q}\in  R_{B,G_{1,1},1} \right\} .$$
Observe that  for  any two distinct balls $L$ and $\tilde{L}$ in $ A(1,B_0,G_{1,1}),$
$$3L\subset 3L^s\subset B  ~\text{and}~ 3L^s \cap 3\tilde{L}^s =\emptyset.$$ Furthermore,
$$ \L(B)\ge \sum_{L\in A(1,B_0,G_{1,1})}\L(L^s) \ge 10^{-7}{\L(B)}.$$

Set
$$\K_1(B_0,1)=\{ L: L \in A(1,B,G_{1,1}), B\in U(1,B_0,1) \} .$$
Thus, the first sub-level is defined as
$$K_1(B_0,1)=\bigcup_{L\in\K_1(B_0,1)}L=\bigcup_{B\in U(1,B_0,1) }\bigcup_{L\in A(1,B_0,G_{1,1})}L.$$

\textbullet \   Step $4_1.$ Volume estimation.

We check at once that
\begin{equation}\label{mh1}
   \L(B_0)\ge\sum_{L\in \K_1(B_0,1) }\L(L^s)= \sum_{L \in A(1,B_0,G_{1,1})}\L(L^s)\ge 10^{-7}\L(B_0),
\end{equation} since $U(1,B_0,1)=\{B_0\}.$
By \eqref{size1} and \eqref{mh1},
\begin{align*}
  \L\left(\bigcup_{B\in U(1,B_0,1) }\bigcup_{L\in A(1,B_0,G_{1,1})}4L\right)=& \sum_{L \in A(1,B_0,G_{1,1})} \L(4L)\\
    =&4\sum_{L \in A(1,B_0,G_{1,1})} \frac{\L(L)}{\L(L^s)}\L(L^s)\\
    \le& 6^{-1} \sum_{L \in A(1,B_0,G_{1,1})}\L(L^s)
    \le \frac{\L(B_0)}{6}.
  \end{align*}

This completes the construction of the first sub-level. Now we construct the next sub-level by induction. Suppose that the sub-levels
$$K_1(B_0,1), \cdots,  K_1(B_0,l-1)$$ have been well constructed. Each of them has the form, saying
$$ K_1(B_0,i)=\bigcup_{L\in \K_1(B_0,i)}L=\bigcup_{B\in U(1,B_0,i)}\bigcup_{L\in A(1,B,G_{1,i}) }L,$$ and
\begin{equation}\label{m1} \sum_{B\in U(1,B_0,i)}\sum_{L\in A(1,B,G_{1,i}) }\L(4L)\le \frac{1}{6^i}\cdot\L(B_0) ,\end{equation} where
$G_{1,i}>  G_{1,i-1}$ for any $2\le i\le l-1.$

\textbullet \ Step $1_l.$ Notation $U(1,B_0,l).$

Since each sub-level of the construction  is expected  to be distinctly separated from the previous ones,  we first verify  that the set
$$ \widetilde{B}_l = \left(\frac{3}{5}B_0\right)\Big\backslash \bigcup_{i=1}^{l-1}\bigcup_{B\in U(1,B_0,i)}\bigcup_{L\in A(B,G_{1,i},1) }4L$$ is large enough. By \eqref{m1},
we conclude that
\begin{align*}
  \L\left(\bigcup_{i=1}^{l-1} \bigcup_{B\in U(1,B_0,i)}\bigcup_{L\in A(B,G_{1,i},1) }4L\right)\le \sum_{i=1}^\infty \frac{\L(B_0)}{6^i}\le \frac{1}{5}\L(B_0).
\end{align*}Thus,
\begin{equation}\label{meq1}
  \L(\widetilde{B}_l)\ge \frac{1}{5} \L(B_0).
\end{equation}

Set $$r=\min\left\{|L|/5: L \in  \K_1(B_0,i),1\leq i\le l-1\right\} \quad \text{and} \quad \F_l=\{B(x,r):x \in \widetilde{B}_l\}.$$ Thus, by Lemma \ref{5rcover}, we get a disjoint sub-collection $U(1,B_0,l)\subset \F_l$ such that
$$\widetilde{B}_l \subset \bigcup_{B\in \F_l} \subset \bigcup_{B \in U(1,B_0,l)}5B. $$ It follows from the definition of $r$ that
$$ \bigcup_{B\in U(1,B_0,l)}B\subset B_0.$$ Hence, it can be easily verified that the collection $U(1,B_0,l) $ is finite. Besides, for any $B\in U(1,B_0,l),$ we have
\begin{equation}\label{clude1}B \cap  3L =\emptyset \quad \text{for any}\ L \in \bigcup_{i=1}^{l-1}\K_1(B_0,i). \end{equation}


In summary, by construction, the collection $U(1,B_0,l)$ satisfies the following properties.

  a) For any $B\neq B' \in U(1,B_0,l),$ \vspace{-2mm}\begin{equation}\label{con2}B\subset B_0\quad \text{and}\quad B\cap B'=\emptyset.\end{equation}

  \vspace{-4mm}

  b)   The balls in $U(1,B_0,l)$ are thoroughly separated  from any  balls in the previous sub-levels. Specifically, \eqref{clude1} holds.

  c) These balls  almost pack $B_0.$ Indeed,
  \begin{equation}\label{meq2}  \sum\limits_{B \in U(1,B_0,l)} \L(B)\ge \frac{1}{25}\L(B_0).\end{equation}  This is directly deduced from \eqref{meq1} and
  $$ \L(\widetilde{B}_l)\le \sum\limits_{B \in U(1,B_0,l)} \L(5B)=5\sum\limits_{B \in U(1,B_0,l)} \L(B).$$

\textbullet \ Step $2_l.$ Use Lemma \ref{lem2}.

Let  $G_{1,l} > G_{1,l-1}$  be sufficiently large so that
\begin{equation}\label{choose1} \frac{\L\left(\C_1\left(\frac{p}{q}, \psi\right)\right)}{\L\left(\C_1^s\left(\frac{p}{q}, \psi\right)\right)}\leq \frac{1}{4 \cdot 6^l} \quad \text{whenever} \quad q\ge G_{1,l},\end{equation}
and
 \begin{equation}\label{small}\left|\C_1^s\left(\frac{p}{q}, \psi\right)\right|\le  \frac{1}{5}\min_{1\leq i \le l-1} \min_{L \in \K_1(B_0,i)} |L^s| \quad \text{whenever} \quad q\ge G_{1,l}.\end{equation}
Let $k=1$ and $Q'=G_{1,l}.$ For each $B\in U(1,B_0,l),$ apply  Lemma \ref{lem2} to obtain a collection of well-separated balls of the form
$$3\C_1^s\left(\frac{p}{q}, \psi\right)\subset B  \quad \text{with} \quad \frac{p}{q}\in R_{B,G_{1,l},1} .$$  Furthermore,
\begin{equation*}\sum_{\frac{p}{q}\in R_{B,G_{1,l},1}} \L\left(\C_1^s\left(\frac{p}{q}, \psi\right)\right) \ge 10^{-7} \L(B), \end{equation*} and
$$3\C_1^s\left(\frac{p}{q}, \psi\right) \cap 3\C_1^s\left(\frac{p'}{q'}, \psi\right)=\emptyset \quad \text{if} \quad \frac{p}{q}\neq \frac{p'}{q'} \in  R_{B,G_{1,l},1}.$$

\textbullet \ Step $3_l.$ Shrinking. Notation $A(B_0,G_{1,l},l)$ and $K_1(B_0,l).$

Remember that for any $\frac{p}{q}\in R_{B,G_{1,l},1}$ and $B\in U(1,B_0,l)$
$$3\C_1\left(\frac{p}{q},\psi\right) \subset  3\C_1^s\left(\frac{p}{q},\psi\right) \subset B .$$
Let $$A(1,B,G_{1,l})=\left\{\C_1\left(\frac{p}{q},\psi\right):  \frac{p}{q}\in R_{B,G_{1,l},1} \right\}.$$
Thus for any $L\neq L' \in A(1,B,G_{1,l}),$
\begin{equation}\label{con1}3L \subset 3L^s \subset B\quad \text{and}\quad 3L^s\cap 3L'^s=\emptyset, \end{equation} and
\begin{equation}\label{m2}\sum_{L\in A(1,B,G_{1,l}) } \L(L^s) \ge 10^{-7} \L(B). \end{equation} In addition, by \eqref{small},  $$ |L^s|\le \frac{1}{5}|M^s| \quad \text{for any} \ M\in \bigcup_{i=1}^{l-1}\K_1(B_0,i). $$

Define the set
$$\K_1(B_0,l)=\{L: L\in A(1,B,G_{1,l}),B\in U(1,B_0, l)\}. $$ Hence
$$K_1(B_0,l)=\bigcup_{L\in\K_1(B_0,l)}L. $$

\textbullet \   Step $4_l.$ Volume estimation.

On the one hand, by \eqref{meq2} and \eqref{m2}, we can assert that
\begin{align}\label{sim2}
 \sum_{L\in \K_1(B_0,l)} \L(L^s) = &\sum_{B\in U(1,B_0,l) } \sum_{L \in A(1,B,G_{1,l})} \L(L^s)\\ \nonumber
  \ge&\sum_{B\in U(1,B_0,l) } 10^{-7}\L(B) \\ \nonumber
  \ge &10^{-7}\cdot25^{-1} \L(B_0). \nonumber
  \end{align}
 On the other hand, combining  \eqref{con2}, \eqref{choose1} and  \eqref{con1}, we get that
 \begin{align}\label{sim} \sum_{B\in U(1,B_0,l)}\sum_{L\in A(1,B,G_{1,l}) }\L(4L)= &4 \sum_{B\in U(1,B_0,l)}\sum_{L\in A(1,B,G_{1,l}) } \frac{\L(L)}{\L(L^s)}\L(L^s)  \\
 \le & \frac{1}{6^l}\sum_{B\in U(1,B_0,l)}\sum_{L\in A(1,B,G_{1,l}) }\L(L^s) \nonumber\\
 \le & \frac{1}{6^l}\sum_{B\in U(1,B_0,l)}\L(B) \nonumber
 \le  \frac{\L(B_0)}{6^l}\nonumber
 .\end{align}

 This finishes the  induction step. Let $l_{B_0}=\lfloor\eta \cdot \L(B_0)^{-1}\rfloor+1.$ Finally,  define
 $$\K_1=\K_1(B_0)=\bigcup_{i=1}^{l_{B_0}}\K_1(B_0,i)  .$$ Thus, the first level of the Cantor set is defined as
 $$K_1=K_1(B_0)=\bigcup_{L\in \K_1(B_0)} L=\bigcup_{i=1}^{l_{B_0}} \bigcup_{L\in \K_1(B_0,i)} L.$$

{\em The general level $K_n$.}

Assume that $\K_1, \cdots ,\K_{n-1}$ have been  well defined, which is composed of collections of sub-balls of each ball in the previous level. Define
$$\K_n=\bigcup_{B_{n-1}\in \K_{n-1}} \K_n(B_{n-1}) \quad  \text{and} \quad K_n=\bigcup_{L\in \K_n}L ,$$
where \begin{align*} \K_n(B_{n-1})=\bigcup_{i=1}^{l_{B_{n-1}}}\K_n(B_{n-1},i)
=\bigcup_{i=1}^{l_{B_{n-1}}}\{L: L\in A(n,B,G_{n,i}), B \in U(n,B_{n-1},i)\}.\end{align*}

The construction of the $n$-level is the same as that of the first level, we sketch the main steps in the construction process. Replace  $B_0$ with $B_{n-1}$ and repeat the construction route of $\K_1(B_0)$ to obtain $\K_n(B_{n-1}).$ Generally, the route is
$$B_{n-1}\stackrel{cover}{---\longrightarrow} B\in U(n,B_{n-1},l)\stackrel{Lemma \ \ref{lem2}}{---\longrightarrow} 3L^s\subset B \stackrel{Shrinking }{---\longrightarrow}L\in A(n,B,G_{n,l}) \ $$

The only remaining point meriting concern is that we apply Lemma \ref{lem2} to each  $B \in U(n,B_{n-1},l)$ with $k=n$ and $Q'=G_{n,l}>G_{n,l-1}$ ($G_{n,0}=G_{n-1,l_{B_{n-2}}}$), which satisfies

\begin{equation}\label{as}\frac{\L\left(\C_n\left(\frac{p}{q}, \psi\right)\right)}{\L\left(\C_n^s\left(\frac{p}{q}, \psi\right)\right)}\leq \frac{1}{4 \cdot 6^l}  \quad \text{whenever} \quad \ q\ge G_{n,l},\end{equation}
and
$$ \left|\C_n^s\left(\frac{p}{q}, \psi\right)\right|\le  \frac{1}{5}\min_{1\leq i \le l-1} \min_{L \in \K_n(B,i)} |L^s|\quad \text{whenever} \quad q\ge G_{n,l}.$$ It is clear that each ball in $\K_n(B_{n-1},l)$ is of the form
$$\C_n\left(\frac{p}{q},\psi\right) \quad \text{where} \ ~\frac{p}{q} \in R_{B,G_{n,l},n}\ ~\text{and} \ ~B\in U(n,B_{n-1},l).$$

For each $B_{n-1}\in \K_{n-1},$ let $$l_{B_{n-1}}=\lfloor{4|B_{n-1}^s|}\cdot{\L(B_{n-1})^{-1}}\rfloor+1.$$
So,
$$K_n=\bigcup_{B_{n-1}\in\K_{n-1}}\bigcup_{i=1}^{l_{B_{n-1}}}\bigcup_{B\in U(n,B_{n-1},i)} \bigcup_{L\in A(n,B,G_{n,l}) }L.$$
Finally, the desired Cantor set is defined as
$$ K_{\eta} =\bigcap_{n\ge 0} K_n= \bigcap_{n\ge0 } \bigcup_{L_n\in \K_n} L_n.$$

\begin{pro}
We have $K_{\eta} \subset E(\psi).$
\end{pro}
\begin{proof}
For any sufficiently small $\epsilon>0,$ there exists $n_0\ge1$ such that $(1-\epsilon)<(1-2^{-n})$ holds for each $n\ge n_0.$ Given a point $x\in K_{\eta },$ naturally, we get $$x \in \bigcap_{n\ge n_0}K_n.$$ By  construction, for each $n\ge n_0,$ there exist $B_{n-1}\in \K_{n-1},$ $1\le l\le l_{B_{n-1}}$ and $B\in U(n,B_{n-1},l)$ such that $x\in \C_n\left(\frac{p}{q},\psi\right)$ for some $\frac{p}{q}\in R_{B,G_{n,l},n},$ i.e.
$$ (1-2^{-n})\psi(q)<x-\frac{p}{q} <\psi(q).$$ Since $q\ge G_{n,l},$ we can assume that $G_{n,l}$ is so large that the  set $R_{B,G_{n,l},n}$ is disjoint from any other  sets  $R_{B',G_{n',l'},n'}$ appearing before. Thus, we can assert that $x\in E(\psi)$ and this completes the proof.
\end{proof}
Next, we summarise the properties of the Cantor set constructed above for later use.
\begin{pro}\label{keypro}
  Let $n\ge1.$
  \begin{enumerate}[(i)]
    \item For any $B_{n-1}\in \K_{n-1},$ we have
        $$3L\subset 3L^s \subset B_{n-1}  \quad \text{and} \quad 3L\cap 3L' = \emptyset , \quad \text{if } L\neq L' \in \K_n(B_{n-1}) .$$
    \item For any $B_{n-1} \in  \K_{n-1}$ and $1\le l \le l_{B_{n-1}},$ the balls in $$\left\{L^s: L \in \K_n(B_{n-1},l)\right\}$$ are well separated and contained in $B_{n-1}.$ Indeed, for any $L\neq \tilde{L} \in \K_n(B_{n-1},l)$ $$ 3L^s\cap3\tilde{L}^s= \emptyset .$$
    \item For any $B_{n-1} \in  \K_{n-1}$ and $1\le l \le l_B,$
    $$\sum_{L\in \K_n(B_{n-1},l)} \L(L^s) \gg \L(B_{n-1}), $$
    where the implied constant in the Vinogradov symbol ($\gg$) is absolute.
    \item For any $B_{n-1} \in  \K_{n-1},$  $1\le l < l_{B_{n-1}},$ and $ L\in \K_{n-1}(B_{n-1},l)$
    $$ |M^s| \le \frac{1}{5} |L^s| \quad \text{if} \quad   M\in \K_{n-1}(B_{n-1},l+1).   $$
    \item The number of local sub-levels  is determined by
    $$   l_B=\left\{\begin{array}{ll}
   \lfloor\eta \cdot \L(B)^{-1}\rfloor+1 &\text{if}\quad B\in \K_0, \\
   \lfloor{4|B^s|}\cdot{\L(B)^{-1}}\rfloor+1 &\text{if}\quad B \in \K_n , n\ge1.
    \end{array}
  \right.$$

  \end{enumerate}
\end{pro}
\begin{proof}We only need to show the items $(i)$ and $(iii),$ since others are trivial by the construction of $K_{\eta}.$  We first prove the item $(i).$ By construction, there are $l,l'\in \{1,\cdots,l_{B_{n-1}}\}$ such that $L\in \K_n(B_{n-1},l)$ and $L'\in \K_n(B_{n-1},l').$ Suppose that $l\le l'.$

Recall that the construction process of $L,L':$
\begin{align*}&B_{n-1}\xrightarrow[]{covering} B\in U(n,B_{n-1},l)\xrightarrow[]{Lemma \ \ref{lem2}} L^s\subset B\xrightarrow[]{Shrinking} L\in  A(n,B,G_{n,l}) \\
&B_{n-1}\xrightarrow[]{covering} B'\in U(n,B_{n-1},l')\xrightarrow[]{Lemma \ \ref{lem2}} L'^s\subset B'\xrightarrow[]{Shrinking} L'\in  A(n,B',G_{n,l'})) .\end{align*} Note that the balls
$$\{B: B\in U(n,B_{n-1},l)\} $$ are disjoint
and the sub-balls
$$\{3L:L\in A(n,B,G_{n,l})\} $$ of $B$ are also disjoint. Thus, if $l=l',$ then $3L\cap 3L' = \emptyset.$  For $l<l',$ it follows from the construction of $U(n,B_{n-1},l')$ that $ 3L\cap B'=\emptyset$ for all $L \in \bigcup_{i=1}^{l'-1}\K_n(B_{n-1},i).$ Thus $3L\cap 3L' = \emptyset$ when $l<l'$ and it establishes item $(i).$

We now turn to prove $(iii).$ By the construction of $U(n,B_{n-1},l),$
\begin{align}\label{rh} \L\left(\frac{3}{5}B_{n-1}\Big\backslash \bigcup_{i=1}^{l-1}\bigcup_{B\in U(n,B_{n-1},i)}\bigcup_{L\in A(n,B,G_{n,i}) }4L \right)&\le \sum\limits_{B \in U(1,B_{n-1},l)} \L(5B) \\
&=5\sum\limits_{B \in U(n,B_{n-1},l)} \L(B).\nonumber
\end{align}Based on \eqref{as} and the same estimation as \eqref{sim}, by induction we obtain
$$ \L \left( \bigcup_{B\in U(n,B_{n-1},l)}\bigcup_{L\in A(n,B,G_{n,l}) }4L \right)\le \frac{1}{6^l}\cdot\L(B) \quad \text{for any } \quad  1\le l \le l_B.$$  Thus, the left-hand side of inequality \eqref{rh} is at least $5^{-1}\L(B).$ Hence,
 \begin{equation}\label{g1}\sum\limits_{B \in U(n,B_{n-1},l)} \L(B) \ge \frac{1}{25}\L(B_{n-1}). \end{equation} Note that \begin{equation}\label{g2}\sum_{L \in A(n,B,G_{n,l})} \L(L^s) \ge \frac{1}{10^{7}}\L(B) .\end{equation} Finally, by \eqref{g1}, \eqref{g2} and the same estimate as  \eqref{sim2}, we get the third item.
\end{proof}

\subsection{Mass distribution}
In this section, we define a probability measure $\nu$ supported on $K_{\eta}.$

For $n=0,$ $\K_0$ contains only one element $L=(0,1).$ Naturally, $\nu(L):=1. $

For $n\ge1$ and a ball $L \in \K_n,$ there exists a unique ball $B_{n-1}\in \K_{n-1}$ and $l \in \{1,\cdots,l_{B_{n-1}}\}$ such that $$ L \in \K_n(B_{n-1},l).$$  Assume that the measure of balls in $\K_{n-1}$  have been well  defined. Then, define
\begin{equation}\label{e1} \nu(L) = \frac{\L(L^s)}{ \sum_{i=1}^{l_{B_{n-1}}} \sum_{\tilde{L}\in \K_n(B_{n-1},i)} \L(\tilde{L}^s)} \cdot \nu(B_{n-1}).\end{equation}

By Kolmogorov's consistency theorem, the set function $\nu$ can be uniquely extended into a probability measure supported on $K_{\eta}$. Actually, for any $A\subset [0,1],$
$$\nu(A)=\nu(A\cap K_{\eta}) :=\inf \left\{\sum_{i\ge1} \nu(L_i): A\subset \bigcup_{i\ge1}L_i, L_i \in \bigcup_{n\ge0}\K_n \right\}.$$

 \subsection{The measure of a ball in $\K_n$}
 Our goal in this section  is to prove that for any $L \in \K_n$ and $n\ge1$ we have
\begin{equation}\label{goali}\nu (L)\ll \frac{|L|^s}{\eta} . \end{equation}

We begin with the first level. By construction, for any $L\in \K_1,$ we have $L \in \K_1(B_0, l)$ for some $l \in \{1,\cdots,l_{B_0}\},$ since $\K_0=\{B_0\}.$ Thus, by \eqref{e1} and $\nu(B_0)=1,$ we get that
\begin{align*}
  \nu(L)=\frac{\L(L^s)}{\sum_{l=1}^{l_{B_0}} \sum_{\tilde{L}\in \K_1(B_0,l)}\L(\tilde{L}^s )}.
\end{align*}Combining the items $(iii)$  and $(v)$ of Proposition \ref{keypro},  we conclude that
\begin{align*}
  \nu(L)\ll \frac{\L(L^s)}{l_{B_0} \cdot\L(B_0)}\ll \frac{|L^s|}{\eta}=\frac{|L|^s}{\eta}.
\end{align*}

We show by induction that \eqref{goali} still holds for all $n\in\N.$ Assume that \eqref{goali} holds for all $L\in \K_{n-1}.$   It is easy to check that for each $L \in \K_n $ there exists unique $B_{n-1} \in \K_{n-1}$ and $l \in \{1,\cdots,l_{B_{n-1}}\}$ such that $ L \in \K_n(B_{n-1},l).$ By definition and our hypothesis,
\begin{equation*}\nu(L)  \ll \frac{\L(L^s)}{ \sum_{l=1}^{l_{B_{n-1}}} \sum_{\tilde{L}\in \K_n(B_{n-1},l)} \L(\tilde{L}^s)} \cdot \frac{|B_{n-1}|^s}{\eta}.\end{equation*} Using Proposition \ref{keypro} $(iii)$  and  Proposition \ref{keypro} $(v),$ we get that
\begin{equation}\label{ele}\nu(L)  \ll \frac{\L(L^s)}{ l_{B_{n-1}} \cdot \L(B_{n-1})} \cdot \frac{|B_{n-1}|^s}{\eta}\ll  \frac{|L|^s}{\eta}.\end{equation} This finishes the inductive step and thereby completes the claim.

 \subsection{Measure of a general ball}
Recall that $|B|$ denotes the radius of a ball $B.$ Let $r_0=\min\{|L|: L\in \K_1 \}.$
In this section, we will show  \eqref{goal} holds for any ball $B(x,r)$ with $r\le r_0,$ i.e.
\begin{equation}\label{fg}\nu(B(x,r)) \ll \frac{r^s}{\eta} .\end{equation} Then by Mass Distribution Principle (Proposition \ref{p1}), we get the desired consequence.

Without loss of generality we can assume that $B(x,r)\cap K_{\eta} \neq \emptyset;$ otherwise $\nu(B(x,r))=0,$ since $\nu$  is supported on $K_{\eta}.$  If for any $n\ge0$ the ball $B(x,r)$  intersects only one element in $\K_n,$ then by \eqref{ele} we have also $\nu(B(x,r))=0.$  Thus, we can assume that there exists a unique  integer $n$ such that $B(x,r)$ intersects exactly one element  in $\K_{n-1},$ which we denote by $B_{n-1},$ and at least two elements in  $\K_n,$ which we denote by $L$ and $M,$ respectively.

By hypothesis and Proposition \ref{keypro} $(i),$ we see that $$ B(x,r)\cap L\neq \emptyset, \quad  B(x,r)\cap M\neq \emptyset \quad \text{and} \quad 3L\cap3M=\emptyset .$$ Thus it follows that
\begin{equation}\label{radius1}|L|< r \quad \text{and} \quad |M|< r.  \end{equation} Hence $n\ge2.$ We can further assume that
\begin{equation}\label{ieqrad}
  r\le |B_{n-1}|,
\end{equation} otherwise it is evident that
$$ \nu(B(x,r))\le \nu(B_{n-1}) \ll\frac{|B_{n-1}|^s}{\eta} \le \frac{r^s}{\eta}.$$

Note that from \eqref{ele} we obtain
\begin{align}\label{f4}\nu(B(x,r)) &\le \sum_{i=1}^{l_{B_{n-1}}}\sum_{\substack{L \in \K_n(B_{n-1},i)\\ L\cap B(x,r) \neq \emptyset}}\nu(L) \\
&\ll \sum_{i=1}^{l_{B_{n-1}}}\sum_{\substack{L \in \K_n(B_{n-1},i)\\ L\cap B(x,r) \neq \emptyset}}\frac{|L|^s}{\eta}. \nonumber
\end{align}  Thus, we argue by distinguishing two cases.

\medskip
{\em Case 1.} Sub-levels $\K_n(B_{n-1},i)$ for which only one element can intersect $B(x,r).$

 Denote $i^*$ by the smallest integer $i$ such that there exists $L \in \K_n(B_{n-1},i )$ satisfying $L\cap B(x,r)\neq \emptyset.$ Naturally, $L\cap B(x,r)=\emptyset$ for any $L\in \K_n(B_{n-1},i)$ with $i<i^*$  and there exists  a unique ball $L^* \in \K_n(B_{n-1},i^*)$ such that $L^*\cap B(x,r)\neq \emptyset.$ Hence, combining \eqref{radius1} with Proposition \ref{keypro} $(iv),$ we conclude that
 \begin{align}\label{f3}
   \sum_{i \in \text{\it Case 1}}\sum_{\substack{L\in \K_{n}(B_{n-1},i)\\ L\cap B(x,r) \neq \emptyset}} \frac{|L|^s}{\eta}& \leq  \sum_{i \in \text{\it Case 1}} \frac{1}{5^{i-i^*}} \frac{|L^*|^s}{\eta} \\
   & \ll \frac{|L^*|^s}{\eta} \leq \frac{r^s}{\eta}.\nonumber
   \end{align}

\medskip
{\em Case 2.} Sub-levels $\K_n(B_{n-1},i)$ for which at least two elements can intersect $B(x,r).$

Since we are in  case $2,$ we can assume that $L, M \in \K_n(B_{n-1},i).$ By
Proposition \ref{keypro} $(ii),$ $$3L^s \cap  3M^s = \emptyset. $$ Keep
in mind that
$$B(x,r)\cap L\neq \emptyset, \quad B(x,r)\cap M \neq \emptyset \quad \text{and} \quad L\subset L^s. $$ Thus,  by a simple geometric observation, we can assert that \begin{equation}\label{f1}  L^s \subset 3B(x,r).\end{equation} For any balls $B,$ it is obvious that $|B|^s=|B^s|$ and $\L(B)  \asymp |B|.$ Then, from \eqref{f1} and Proposition \ref{keypro} $(v),$ we get that
\begin{align}\label{f2}
  \sum_{i\in \text{Case } 2} \sum_{\substack{L\in \K_n(B_{n-1},i)\\L\cap B(x,r)\neq \emptyset}} \frac{|L^s|}{\eta}\ll \sum_{i\in \text{Case } 2} \frac{\L(B(x,r))}{\eta}& \ll l_{B_{n-1}}\cdot\frac{r}{\eta}\\&
  \ll \frac{|B_{n-1}|^s}{|B_{n-1}|} \cdot \frac{r}{\eta} \le \frac{r^s}{\eta}. \nonumber
\end{align} The last inequality follows from the function $x\mapsto x^{s-1}$ is decreasing and \eqref{ieqrad}.

 Combining \eqref{f4},\eqref{f3} and \eqref{f2}, we obtain the desired  result \eqref{fg}.

\begin{rem}
    An alternative proof of the divergence  of the Hausdorff measure of the set  $E(\psi)$ can  be deduced by combining a
 variation of inhomogeneous Diophantine approximation (see Theorem 2.1 of \cite{HW25}) with Cassels' scaling lemma.
\end{rem}

\section{Proof of Theorem \ref{mainthm}}
In this section, we will focus  on the proof of Theorem \ref{mainthm}. The upper bound for the Hausdorff dimension of $E(\psi_1) \times \cdots \times E(\psi_n)$ follows easily from  Theorem \ref{thm2}, since $E(\psi_i)\subset W(\psi_i)$ and $W(\psi_i)$ is dense in $\R.$ The main difficulty lies in determining the lower bound. To this end, at first, we construct a Cantor subset $F_\infty$; secondly, define a proper mass distribution $\mu$ supported on $F_\infty;$
thirdly, estimate the $\mu$ measure of the general ball; finally, we obtain the expected lower bound by applying the Mass Distribution Principle.

\subsection{Cantor subset}


We sketch the main ideas about the construction of the Cantor subset. In what follows, a ball $B(\x,r)\subset [0,1]^n$ is of the form $$B(\x,r)=\prod_{i=1}^n B(x_i,r)\ \text{for} \ \x=(x_1,\cdots, x_n)\in [0,1]^n, $$
      where $B(x_i ,r)$ is a ball in $[0,1]$ and  is called the $i$-th direction of the ball $B(\x,r).$ Recall that $|B|$ denotes the radius of a ball $B.$
\begin{itemize}
  \item We begin with $[0,1]^n.$ Applying Lemma \ref{keylem} to $(0,1)$ in the first direction of $[0,1]^n$ yields a family of rational numbers
      $$\frac{p_{1,i}}{q_{1,i}}\in J_{[0,1]}.$$ For each $\frac{p_{1,i}}{{q_{1,i}}},$  we  obtain
      $$\C_1\left( \frac{p_{1,i}}{q_{1,i}},\psi_1\right).$$
  \item We lift these balls in the first direction of $[0,1]^n$ to rectangles in $[0,1]^n.$
  \item Divide these rectangles into balls $B_1$ in $[0,1]^n$. This yields the first level of the desired Cantor subset.
  \item For the second level, we repeat the same steps for each ball in the first level. The only difference is that we apply Lemma \ref{keylem} to the ball  in the second direction of $B_1.$
\end{itemize}

We now give the construction of the Cantor set in detail.
Keep in mind that functions $\psi_l:\R_{> 0}\rightarrow \R_{> 0} $ ($1\le l \le n $) are non-increasing and satisfy $\psi_l(x)=o(x^{-2}).$ Thus there exists an integer $Q_0$ such that $ q^2\psi_l(q)<1/1000$ for all $q\ge Q_0 $ and all $1\le l \le n.$ From now on, $\psi_0$ stands for $\psi_n.$

Let $\epsilon >0 $ be given. We inductively construct a  strictly increasing sequence $\{Q_k\}_{k\ge1}$ of positive integers that grows rapidly.  Set
\begin{align}\label{measue1} Q_1 > \max\{20000\log2, 9Q_0 \} \ \text{and}
\ -\log \psi_1(Q_1)\le \frac{2\lambda_1}{2-\epsilon\cdot\lambda_1}\cdot  \log Q_1.\end{align}
Suppose that $k=tn+l>1$ for some integer $t\ge 0$ and $1\leq l \leq n,$ let
\begin{equation}\label{c1}Q_k>\max \{2^{17+t} \psi_{l-1}(Q_{k-1})^{-1}\log (2^{t+2}\psi_{l-1}(Q_{k-1})^{-1}),Q_{k-1}\}\end{equation} and
\begin{align}\label{ms2}-\log \psi_{l}(Q_k) \le \frac{2\lambda_{l}}{2-\epsilon\cdot\lambda_{l}} \cdot \log Q_k
.\end{align}

{\em Level $1$ of the Cantor set.}

  \textbullet \ Step 1: Apply Lemma \ref{keylem}. Due to \eqref{measue1}, we can apply lemma \ref{keylem} to $(0,1)$ and $Q_1.$ Then, we get  a family of  $\frac{1}{2Q_1}$-separated balls in $[0,1]$:
$$\left\{\C_1\left(\frac{p_{1,i}}{q_{1,i}},\psi_1\right)\right\}_{1\leq i \leq \left\lfloor{160^{-1}} { Q_1^2} \right\rfloor} \quad \text{with} \quad \frac{p_{1,i}}{q_{1,i}}\in J_{(0,1)}.$$
Since we are constructing subsets of $[0,1]^n,$ we lift these balls to $[0,1]^n.$ For each $\frac{p_{1,i}}{q_{1,i}}\in J_{(0,1)},$ we define

\begin{align}\label{R1}R_{1,i}&=\C_1\left(\frac{p_{1,i}}{q_{1,i}},\psi_1\right) \times \underbrace{[0,1] \times \cdots \times[0,1]}_{n-1}\\
 &=:B(x_{1,i},r_{1,i})\times B(x_{2},r) \times \cdots \times B(x_{n},r)  \nonumber.\end{align} Keep in mind that for each $1\le i \le \left\lfloor{160^{-1}} { Q_1^2} \right\rfloor,$
\begin{equation}\label{sq1} \frac{ Q_1}{9} \le q_{1,i} \le Q_1\quad \text{and} \quad  r_{1,i}= \left| \C_1\left(\frac{p_{1,i}}{q_{1,i}},\psi_1\right)\right|=\frac{\psi_1(q_{1,i})}{4} .\end{equation} Owing to \eqref{sq1}, the choice of $Q_0$ and the monotonicity of $\psi_1,$  one has
\begin{equation}\label{new1}\frac{\psi_1(Q_1)}{4}\le r_{1,i}\le \frac{\psi_1({Q_1}/{9})}{4}<\frac{1}{20Q_1^2} .\end{equation}
Combining \eqref{new1} with the separation property of $\left\{B(x_{1,i},r_{1,i})\right\} ,$ we conclude that for any $i\neq j$
\begin{equation}\label{sq2}
  3B(x_{1,i},r_{1,i})\cap 3B(x_{1,j},r_{1,j}) =\emptyset.
\end{equation}
Denote the collection of these rectangles by $R(1,[0,1]^n)$ and the index set of these rectangles by $R_*(1,[0,1]^n)$.

We account for subscripts in the set $A_{1,i},$  the real numbers $r_{1,i}$ and the point $x_{1,i}$:  $i$  indicates  an integer in the index set $R_*(1,[0,1]^n)$; the first $1$ in the set $A_{1,i}$ and $r_{1,i}$  indicate that we are at the first level of the Cantor set, however the first $1$ in  $x_{1,i}$ indicate that we are in the first direction of $[0,1]^n.$

\textbullet \ Step 2: Division.

  Set $r_{1,i}:=\frac{\psi_1(q_{1,i})}{4}.$
  For each $\frac{p_{1,i}}{q_{1,i}} \in J_{(0,1)},$ we divide the lifted rectangle $R_{1,i}$ in \eqref{R1} into balls of radius $r_{1,i}.$
  Indeed, during the cutting process, we keep the ball in the first direction of $R_{1,i}$ unvaried, since $r_{1,i}$ is just the radius of the ball in the first direction of $R_{1,i}.$

  Thus,  by Lemma \ref{5rcover}, we get a  family  of  $5r_{1,i}$-separated balls  with the form
  \begin{equation}\label{Q1} B_{1,i}=B_{1,i}(\z):=B(z_{1}, r_{1,i})\times B(z_{2}, r_{1,i})\times \cdots \times B(z_{n}, r_{1,i}),\end{equation}  where
   $B(z_1, r_{1,i})=\C_1\left(\frac{p_{1,i}}{q_{1,i}},\psi_1\right).$  Denote the collection of these balls by $A(1,[0,1]^n,i)$ and their centers by $A_*(1,[0,1]^n,i)$.

   By the definition of $A(1,[0,1]^n,i),$ we get that for any two distinct balls $B_{1,i}(\z),B_{1,i}(\z') $ in $A(1,[0,1]^n,i)$
  \begin{equation}\label{sq3} 3B_{1,i}(\z) \cap 3B_{1,i}(\z') =\emptyset, \end{equation}
   and
   $$ \# A(1,[0,1]^n,i) \cdot r^n_{1,i}\asymp \sum_{B_{1,i}(\z) \in A(1,[0,1]^n,i) } \L (B_{1,i}(\z)) \asymp \L(R_{1,i} )\asymp r_{1,i} .$$
Thus, $$\# A(1,[0,1]^n,i) \asymp \frac{1}{r_{1,i}^{n-1}}  .$$

   Set
   $$\F_1=\left\{\prod_{l=1}^nB(z_{l},r_{1,i}):  \z \in A_*(1,[0,1]^n,i),i\in R_*(1,[0,1]^n) \right\} .$$
   Then, the first level of the Cantor set is defined as
   $$F_1=\bigcup_{B\in \F_1 } B=\bigcup_{i\in R_*(1,[0,1]^n)} \bigcup_{\z \in A_*(1,[0,1]^n,i) } \prod_{l=1}^nB(z_{l,i},r_{1,i}) .$$
   In summary,
   \begin{itemize}
     \item The balls in $\F_1$ are all  of the form \eqref{Q1}. For any two different  balls $B, B'\in \F_1,$  by \eqref{sq2} and \eqref{sq3}, we have
         $$3B\cap 3B'=\emptyset .$$ In particular, the balls in the first direction of any two distinct rectangles $R_{1,i}$ and $R_{1,j}$ in $\F_1$ are $\frac{1}{2Q_1^2}$-separated.

     \item  The radius of any ball in $\F_1$ is at most $\frac{1}{4}\psi_1(\frac{Q_1}{9})$ and at least $\frac{1}{4}\psi_1(Q_1).$ Moreover, the balls contained in $R_{1,i}$ have the same radius $r_{1,i}.$
     \item For any $i \in R_*(1,[0,1]^n),$ we have    $$  \# A(1,[0,1]^n,i)\asymp \frac{1}{r_{1,i}^{n-1}}  .$$ 

   \end{itemize}

{\em Level $2$ of the Cantor set.}
The construction of the second level follows the same way as that of the first level. Here, the only distinction is that  Lemma \ref{keylem} is applied  to the ball in the second direction of a generic ball $B_1\in\F_1.$ Given a  ball
$$B_1=\prod_{i=1}^n B(x_i,r)\in \F_1.$$

 \textbullet \ Step 1: Apply Lemma \ref{keylem}.
According to  the construction of the first level and \eqref{c1}, we get that
 $$Q_2\ge \max \{10000 r^{-1}\log (r^{-1}),9Q_0\} .$$ Then, applying Lemma \ref{keylem} to $B(x_2,r)$ and $Q_2,$  we get a family of $\frac{1}{2Q_2^2}$-separated balls in $B(x_2,r):$
$$\left\{\C_1\left(\frac{p_{2,i}}{q_{2,i}}, \psi_2\right)\right\}_{1\leq i \leq \lfloor{ r\cdot Q_1^2}/{80}\rfloor}\quad \text{with} \quad \frac{p_{2,i}}{q_{2,i}}\in J_{B(x_2,r)},$$ whose radius is at most $\frac{1}{4}\psi_{2}(Q_2/9)$ and at least $\frac{1}{4}\psi_{2}(Q_2).$

For each $\frac{p_{2,i}}{q_{2,i}}\in J_{B(x_2,r)},$ define the set
\begin{equation}\label{R2} R_{2,i}=B(x_1,r)\times \C_1\left(\frac{p_{2,i}}{q_{2,i}},\psi_2\right) \times \cdots \times B(x_n,r). \end{equation} Denote the collection of these rectangles by $R(2,B_1)$ and the index set of these rectangles by $R_*(2,B_1)$.

\textbullet \ Step 2: Division.

Let $r_{2,i}=\frac{\psi_2(q_{2,i})}{4}.$ Cut the rectangle in \eqref{R2} into balls of radius $r_{2,i}.$ Keep in mind that the ball in the second direction is invariant. By Lemma \ref{5rcover}, we obtain  a  collection of $5r_{2,i}$-separated balls of radius $r_{2,i}$:
  \begin{equation*}\label{Q2} B_{2,i}:=B(z_1, r_{2,i})\times B(z_2, r_{2,i})\times \cdots \times B(z_n, r_{2,i}),\end{equation*}  where
   $B(z_2, r_{2,i})=\C_1\left(\frac{p_{2,i}}{q_{2,i}},\psi_2\right).$ Denote by $A(2,B_1,i)$ the collection of these balls   and $A_*(2,B_1,i)$ for their centers.

   The second level is composed of collections of subsets of each
ball $B_1\in \F_1.$ Define
   $$ \F_2=\left\{\prod_{l=1}^nB(z_l,r_{2,i}): i\in R_*(2,B_1),\  \z \in A_*(2,B_1,i), \ B_1 \in \F_1 \right\} .$$
   Hence, the second level of the Cantor set is set as
   $$F_2=\bigcup_{B\in \F_2 } B =\bigcup_{B_1\in \F_1 }\bigcup_{i\in  R_*(2,B_1)} \bigcup_{\z \in A_*(2,B_1,i) } \prod_{l=1}^nB(z_l,r_{2,i}).$$

{\em From level $k-1$ to  level $k$.}

It is without loss of generality to assume that $k=tn+l$ for some integer $t\ge 0$ and $1\leq l \leq n.$ Suppose  that $\F_{k-1}$ has been defined and  the radius of each ball in $\F_{k-1}$ is at least $2^{-t-2}\psi_{l-1}(Q_{k-1}),$  let
$$B_{k-1}=\prod_{i=1}^n B(x_i,r) \in \F_{k-1}.$$

 \textbullet \ Step 1: Apply lemma \ref{keylem}.
 It is easily seen that
 $$Q_k\ge \max \{10000 r^{-1}\log (r^{-1}),9Q_0\} .$$  Applying lemma \ref{keylem} to the ball in the $l$-th direction of $B_{k-1}$ and $Q_k,$ we obtain a collection of $\frac{1}{2Q_k^2}$-separated rectangles
 \begin{align*}
  R_{k,i}&=B(x_1,r)\times \cdots \times B(x_{l-1},r) \times \C_{t+1}\left(\frac{p_{k,i}}{q_{k,i}},\psi_{l}\right) \times  \cdots \times B(x_n,r).
 \end{align*}  Denote the collection of these rectangles by $R(k,B_{k-1})$ and the index set of these rectangles by $R_*(k,B_{k-1}).$ Since $\psi_l:\R_{>0}\rightarrow \R_{>0}$ is non-increasing, for each $i\in R_*(k,B_{k-1})$ we have
  $$ Q_k/9 \le q_{k,i}\le Q_k, \ \frac{\psi_{l}(Q_k)}{2^{t+2}}\le \left| \C_{t+1}\left(\frac{p_{k,i}}{q_{k,i}},\psi_{l}\right)\right|\le \frac{\psi_{l}(Q_k/9)}{2^{t+2}}   ,$$  and
  $$\# R(k,B_{k-1}) =\left \lfloor\frac{Q_k^2 |B_{k-1}|}{80} \right\rfloor .$$

\textbullet \ Step 2: Division.

Let $r_{k,i}=\frac{\psi_{l}(q_{k,i})}{2^{t+2}}.$
We can obtain a  collection of $5r_{k,i}$-separated balls
$$ B_{k,i}=B(z_1,r_{k,i})\times \cdots \times B(z_n,r_{k,i}),$$ by cutting the rectangle $R_{k,i}$ into balls of radius  $r_{k,i}$ and  Lemma \ref{5rcover}. Denote the collection of these balls by $A(k,B_{k-1},i)$ and their centers by $A_*(k,B_{k-1},i).$

Set
   $$ \F_k=\left\{\prod_{l=1}^nB(z_l,r_{k,i}): i\in R_*(k,B_{k-1}),\  \z \in A_*(k,B_{k-1},i),\ B_{k-1}\in \F_{k-1} \right\} .$$ Then, we define $F_k$ as follows:
$$F_{k}=\bigcup_{B\in \F_{k} } B =\bigcup_{B_{k-1}\in \F_{k-1}}\bigcup_{i\in  R_*(k,B_{k-1})} \bigcup_{\z \in A_*(k,B_{k-1},i) } \prod_{l=1}^nB(z_l,r_{k,i}).$$

 The desired Cantor subset is taken to be
   $$ F_{\infty} =\bigcap_{k\ge1}F_k. $$

Let $k=tn+l$ for some integer $t\ge0$ and $1\le l \le n.$
 By using the same method as in the first level, one can prove that $\F_k$
 satisfies the following properties.
   \begin{itemize}
   \item  For any two different  balls $B, B'\in \F_k,$  we have
         $$3B\cap 3B'=\emptyset .$$ In particular, for any two distinct rectangles $R_{k,i}$, $R_{k,j}$ in $R(k,B_{k-1})$,
        the balls  in the $l$-th direction of  $R_{k,i},$  $R_{k,j}$ are $\frac{1}{2Q_k^2}$-separated.
     \item The radius of any ball in $\F_k$ is at least $2^{-(t+2)}\psi_{l}(Q_k)$ and  at most $2^{-(t+2)}\psi_{l}(Q_k/9)< (2Q_k)^{-2}.$ Furthermore, the balls contained in $R_{k,i}$ have the same radius $r_{k,i}.$
     \item For any $B_{k-1}\in \F_{k-1}$ and $i\in R_*(k,B_{k-1}) ,$    $$ \# A(k,B_{k-1},i)\asymp\frac{|B_{k-1}|^{n-1}}{r_{k,i}^{n-1}} .$$ 

   \end{itemize}

\begin{pro}
One has
$$ F_{\infty}\subset E_1(\psi_1)\times  \cdots \times E_n(\psi_n) .$$

\end{pro}
  \begin{proof}
   Let $\x=(x_1,\cdots,x_n)\in F_{\infty}.$ The statement will be proved once we prove $x_1 \in E(\psi_1).$ Set $\{k_j\}_{j\ge0}=\{jn+1\}_{j\ge0}.$ Since $\x \in F_{k_j},$ there exist $B_{k_j-1}=\prod_{i=1}^n B(x_i,r) \in \F_{k_j-1}$ and $R_{k_j,l} \in R(k_j,B_{k_j-1})$ such that
     $$\x \in R_{k_j,l}= C_{j+1}\left(\frac{p_{k_j,l}}{q_{k_j,l}},\psi_1\right) \times B(x_2,r) \times \cdots \times B(x_n,r)  .$$ Thus
     $$  c_{j+1}\psi_1(q_{k_j,l}) < \left| x_1 -\frac{p_{k_j,l}}{q_{k_j,l}}\right| < \psi_1(q_{k_j,l})  \ \text{and} \ \frac{Q_{k_j}}{9} \le q_{k_j,l}\le Q_{k_j} .$$  Since $\{c_j\}_{j\ge 1}$ and $\{Q_j\}_{j\ge 1}$ are strictly monotonically increasing, it yields $x_1\in E(\psi_1).$

  \end{proof}

\subsection{Mass distribution}

Now, we define a probability measure  $\mu$ on $F_{\infty}$ in a natural way.
Given a  rectangle $R_{k,i},$   we distribute the mass of $R_{k,i}$ equally among the balls $B_{k,i},$ since the balls $B_{k,i}$  in $R_{k,i}$  share the same radius; the big rectangles $R_{k,i}$ are disjoint, so their mass will be defined according to the number of elements.

{\em Measure of balls in $\F_1$.} For any $R_{1,i} \in R(1,[0,1]^n)$ and $B_{1,i} \in A(1,[0,1]^n,i ),$ we define
\begin{align*}
  \mu(R_{1,i}) &=\left\lfloor \frac{Q_1^2}{160} \right\rfloor^{-1}\asymp \frac{1}{Q_1^2}, \nonumber \\
  \mu(B_{1,i}) &= \frac{1}{\# A(1,[0,1]^n,i)} \cdot \mu(R_{1,i}) \nonumber\\ &\asymp |B_{1,i}|^{n-1}\cdot\mu(R_{1,i}) \asymp \frac{|B_{1,i}|^{n-1}}{Q_1^2} .
\end{align*}

{\em Measure of balls in $\F_k$.} Suppose that  $k=tn+l$ for some integer $t\ge 0,$ $1\leq l \leq n,$ and the measure of balls in $\F_{k-1}$  have been properly defined.
For any $B_{k,i} \in \F_k,$ letting $B_{k-1}\in \F_{k-1}$ be its  predecessor ball and $B_{k,i} \subset R_{k,i} ,$ we define
\begin{align*}
  \mu(R_{k,i}) &=\left \lfloor\frac{Q_k^2 |B_{k-1}|}{80} \right\rfloor^{-1 }\cdot\mu (B_{k-1}) \asymp\frac{\mu (B_{k-1})}{|B_{k-1}|\cdot Q_k^2},\\
  \mu(B_{k,i})&= \frac{1}{\# A(k,B_{k-1},i)} \cdot \mu(R_{k,i}) \\
  &\asymp  \frac{|B_{k,i}|^{n-1}}{|B_{k-1}|^{n-1}}\cdot \mu(R_{k,i})\asymp \frac{|B_{k,i}|^{n-1}}{Q_k^2}\cdot \frac{\mu(B_{k-1})}{|B_{k-1}|^n}.
\end{align*}

According to Kolmogorov's consistency theorem, the set function $\mu$ can be uniquely extended into a probability measure supported on $F_{\infty}$.

\subsection{Local dimension of $\mu$}

In this section, we will  compare the measure of a general ball with its radius. Subsequently, by applying the Mass Distribution Principle,  we can obtain the lower bound of $\hdim F_{\infty}.$

Recall that $\lambda_i= \liminf \limits_{x\rightarrow \infty} \frac{-\log \psi_i (x)}{ \log x} $ where $1\le i \le n.$
Write
$$ s =\min \left\{\frac{2}{\lambda_i}, \ 1\le i \le n\right\} =\min \{\hdim W(\psi_i), \ 1\le i \le n \}.$$

We begin by focusing on the ball in $\F_k.$ For each $1\leq l \leq n,$ write $\{k_t\}_{t\ge0}$ for the sequence $\{tn+l\}_{t\ge0}.$  Then, for each $1\le l \le n,$ by \eqref{c1} and \eqref{ms2}, we get that
\begin{align*}
  \liminf\limits_{t\rightarrow \infty} \frac{2\log Q_{k_t}}{(t+2)\log 2-\log \psi_{l}(Q_{k_t}) }&\ge \liminf\limits_{t\rightarrow \infty}\frac{2\log Q_{k_t}}{-\log \psi_{l}(Q_{k_t}) }\\
  &\ge \frac{2}{\lambda_{l}}-\epsilon. \nonumber
\end{align*}It follows that  there exists integer $k_0$ such that for any $k\ge k_0$ and any $B_{k,i}\in\F_k ,$
\begin{equation}\label{fieq}
  Q_k^{-2}\le |B_{k,i}|^{s-2\epsilon },
\end{equation}
since $|B_{k,i}|\ge2^{-t-2} \psi_l(Q_k) $ if $k=tn+l$ for some $t\ge0$ and $1\le l\le n.$ Hence, for any  $B_{k,i}$ with $k\ge k_0,$ we can assert that
\begin{equation}\label{mj1}\mu(B_{k,i})\ll {Q_k^{-2}}\cdot {|B_{k,i}|^{n-1}} \cdot \frac{\mu(B_{k-1})}{|B_{k-1}|^n} \le {Q_k^{-2}}\cdot {|B_{k,i}|^{n-1}} \le |B_{k,i}|^{n-1+s-2\epsilon}.\end{equation}



We can now  return to consider a general ball $B(\x,\delta)$  with $\x \in F_{\infty}$ and $\delta<\min\{|B_{k_0,i}|: B_{k_0,i}\in\F_{k_0}\}$. Let $k$ be the integer such that the ball $B(\x,\delta)$ can intersect only one ball in $\F_{k-1}$ but at least two balls in
$\F_{k}$. Suppose that $B_{k-1}=\prod_{i=1}^nB(z_i,r)$ is the unique ball in $\F_{k-1}$ which $B(\x,r)$ can intersect. Note that for any two different  balls $B, B'\in \F_k,$  we have
$$3B\cap 3B'=\emptyset .$$  Thus by the choice of $\delta$ and $k,$ we conclude that $k>k_0.$

Without loss of generality, we can assume that $\delta\leq r,$ since otherwise it is immediate by \eqref{mj1} that
$$\mu(B(\x,r))\leq \mu (B_{k-1})\ll r^{n-1+s-2\epsilon} \le\delta^{n-1+s-2\epsilon}.  $$
We only need to consider the following two cases.

{\em Case 1.} The ball $B(\x,\delta)$ can intersect at least  two rectangles $R_{k,m}, R_{k,n}$ in $B_{k-1}$ at the $k$-th level. Without loss of generality we can assume that
$$ R_{k,m} =\prod_{i=1}^{l-1}B(z_i,r) \times B(y_{l,m},  r_{k,m})\times \prod_{i=l+1}^n B(z_i,r),$$
$$ R_{k,n} =\prod_{i=1}^{l-1}B(z_i,r) \times B(y_{l,n},  r_{k,n})\times \prod_{i=l+1}^n B(z_i,r),$$
where $m,n\in R_*(k,B_{k-1})$  and balls $B(y_{l,m},  r_{k,m}),$  $B(y_{l,n},  r_{k,n})$ are $\frac{1}{2Q_k^2}$-separated. Furthermore, it is clear from the choice of $Q_k$ that
\begin{equation*}
   \max\{r_{k,m},r_{k,n}\}<\frac{1}{4Q_k^2}.
\end{equation*}Thus, it follows  that
 \begin{equation}\label{ieq} \delta \ge \frac{1}{4Q_k^2} > \max\{ r_{k,m},  r_{k,n}\}.\end{equation} So,
 $$B(y_{l,m}, r_{k,m})\subset B(x_{l},3\delta)\ \text{and} \ B(y_{l,m}, r_{k,n})\subset B(x_{l},3\delta),$$ where $B(x_{l},3\delta)$ is $3$-times scaling  of the  ball in the $l$-th direction of $B(\x,\delta).$
 Thus, we  conclude that
 $$\#\{R_{k,i} \subset B_{k-1}: R_{k,i}\cap  B(\x,\delta)\neq \emptyset \}\ll \delta \cdot Q_k^2. $$

Keep in mind that each ball $B_{k,m}$ in $R_{k,m}$ has same radius $r_{k,m}$ and for any two different  balls $B_{k,m}, B_{k,m}'\in R_{k,m},$  we have
         $$3B_{k,m}\cap 3B_{k,m}'=\emptyset .$$
 Hence, by \eqref{ieq}, we can assert that
 $$\#\{ B_{k,m}\subset R_{k,m}: B_{k,m} \cap B(\x,\delta) \neq \emptyset \} \ll \left(\frac{\delta}{r_{k,m}} \right)^{n-1} . $$

 Therefore, we get that
 \begin{align*}
   \mu(B(\x,\delta) \cap R_{k,m}) &\leq
   \sum_{B_{k,m}\subset R_{k,m}, B_{k,m}\cap B(\x,\delta)\neq \emptyset} \mu (B_{k,m})\\
   &\ll  \sum_{B_{k,m}\subset R_{k,m}, B_{k,m}\cap B(\x,\delta)\neq \emptyset}\frac{r_{k,m}^{n-1}} {Q_k^{2}} \cdot \frac{\mu(B_{k-1})}{|B_{k-1}|^n} \\
&\ll  \left(\frac{\delta}{r_{k,m}} \right)^{n-1}\cdot  \frac{r_{k,m}^{n-1}}{Q_k^{2}} \cdot \frac{\mu(B_{k-1})}{|B_{k-1}|^n}\\
   & =\frac{\delta^{n-1}}{Q_k^2}  \cdot \frac{\mu(B_{k-1})}{|B_{k-1}|^n}.
 \end{align*}
Thus,
 \begin{align*}
      \mu(B(\x,\delta) \cap B_{k-1}) &\leq \sum_{R_{k,i}\subset B_{k-1}, R_{k,i}\cap B(\x,\delta)\neq \emptyset} \mu (B(\x,\delta) \cap R_{k,i} )\\
      &\leq  \sum_{R_{k,i}\subset B_{k-1}, R_{k,i}\cap B(\x,\delta)\neq \emptyset}\frac{\delta^{n-1}}{Q_k^2}  \cdot \frac{\mu(B_{k-1})}{|B_{k-1}|^n}\\
      &\ll \delta \cdot Q_k^2 \cdot \frac{\delta^{n-1}}{Q_k^2}  \cdot \frac{\mu(B_{k-1})}{|B_{k-1}|^n}\\
   &= \delta ^{n} \cdot \frac{\mu(B_{k-1})}{|B_{k-1}|^n}.
 \end{align*}
 Finally, by \eqref{mj1} and $\delta \le r=|B_{k-1}|$,
 $$ \mu (B(\x,\delta) ) \leq  \delta ^{n} \cdot {r^{-n}} \cdot {r^{n-1+s-2\epsilon}}\le \delta ^{n-1+s-2\epsilon}.$$

 {\em Case 2.} The ball $B(\x,\delta)$ can only intersect one  rectangle $$R_{k,m}=\prod_{i=1}^{l-1}B(z_i,r) \times B(y_{l,m},  r_{k,m})\times \prod_{i=l+1}^n B(z_i,r)$$ of $B_{k-1}$ in the $k$-th level. Recall that the balls in $R_{k,m}$ are of the same radius $r_{k,m}$ and for  any two different  balls $B_{k,m}, B_{k,m}'\in R_{k,m},$  we have
         $$3B_{k,m}\cap 3B_{k,m}'=\emptyset .$$
Since $B(\x,\delta)$ can intersect two  balls in $\F_{k-1},$ thus
\begin{equation}\label{ieq3} \delta> r_{k,m}.\end{equation}

 Hence, with the same estimation as in  case 1, we conclude that

 \begin{equation*}
   \mu(B(\x,\delta))=\mu(B(\x,\delta)\cap R_{k,m})\ll \frac{\delta^{n-1}}{Q_k^2}  \cdot \frac{\mu(B_{k-1})}{|B_{k-1}|^n}\\
   \leq    \frac{\delta^{n-1}}{Q_k^2}.
 \end{equation*} Combining  \eqref{fieq} with \eqref{ieq3}, we conclude that
 $$\mu(B(\x,\delta)) \le \delta^{n-1}\cdot r_{k,m}^{s-2\epsilon} \leq \delta^{n-1+s-2\epsilon} .$$

 Applying the Mass Distribution Principle (Proposition \ref{p1}),  we get that
 $$ \hdim E(\psi_1) \times \cdots \times E(\psi_n) \ge \hdim F_\infty \ge n-1+s-2\epsilon.$$ By the arbitrariness of $\epsilon,$ we conclude that
 $$ \hdim E(\psi_1) \times \cdots \times E(\psi_n) \ge n-1+s.$$ The proof of Theorem \ref{mainthm} is complete.


\end{document}